\documentclass[12pt,leqno,draft]{article}
\usepackage{amsfonts}
\usepackage{amsmath, amsthm, amsfonts, amssymb, color}
\usepackage{mathrsfs}
\usepackage{color}
\usepackage{stmaryrd}
\newtheorem{thm}{Theorem}[section]

\newtheorem{exa}[thm]{Example}
\newtheorem{rem}[thm]{Remark}
\theoremstyle{definition}
\newtheorem{defn}{Definition}[section]
\newcommand{\scr}[1]{\mathscr #1}
\definecolor{wco}{rgb}{0.5,0.2,0.3}

\numberwithin{equation}{section} \theoremstyle{remark}

\newcommand{\ua}{\uparrow}

\title{{\bf   Harnack Inequality and Gradient Estimate for $G$-SDEs with Degenerate Noise}\footnote{Supported in
 part by  NNSFC (11801406).} }
\author{
{\bf     Xing Huang $^{1,a)}$, Fen-Fen Yang $^{1,b)}$   }\\
\footnotesize{  1)Center for Applied Mathematics, Tianjin University, Tianjin 300072, China}\\
\footnotesize{ $^{a)}$xinghuang@tju.edu.cn, $^{b)}$yangfenfen@tju.edu.cn}}
\begin{document}
\allowdisplaybreaks
\def\R{\mathbb R}  \def\ff{\frac} \def\ss{\sqrt} \def\B{\mathbf
B}
\def\N{\mathbb N} \def\kk{\kappa} \def\m{{\bf m}}
\def\ee{\varepsilon}\def\ddd{D^*}
\def\dd{\delta} \def\DD{\Delta} \def\vv{\varepsilon} \def\rr{\rho}
\def\<{\langle} \def\>{\rangle} \def\GG{\Gamma} \def\gg{\gamma}
  \def\nn{\nabla} \def\pp{\partial} \def\E{\mathbb E}
\def\d{\text{\rm{d}}} \def\bb{\beta} \def\aa{\alpha} \def\D{\scr D}
  \def\si{\sigma} \def\ess{\text{\rm{ess}}}
\def\beg{\begin} \def\beq{\begin{equation}}  \def\F{\scr F}
\def\Ric{\text{\rm{Ric}}} \def\Hess{\text{\rm{Hess}}}
\def\e{\text{\rm{e}}} \def\ua{\underline a} \def\OO{\Omega}  \def\oo{\omega}
 \def\tt{\tilde} \def\Ric{\text{\rm{Ric}}}
\def\cut{\text{\rm{cut}}} \def\P{\mathbb P} \def\ifn{I_n(f^{\bigotimes n})}
\def\C{\scr C}   \def\G{\scr G}   \def\aaa{\mathbf{r}}     \def\r{r}
\def\gap{\text{\rm{gap}}} \def\prr{\pi_{{\bf m},\varrho}}  \def\r{\mathbf r}
\def\Z{\mathbb Z} \def\vrr{\varrho} \def\ll{\lambda}
\def\L{\scr L}\def\Tt{\tt} \def\TT{\tt}\def\II{\mathbb I}
\def\i{{\rm in}}\def\Sect{{\rm Sect}}  \def\H{\mathbb H}
\def\M{\scr M}\def\Q{\mathbb Q} \def\texto{\text{o}} \def\LL{\Lambda}
\def\Rank{{\rm Rank}} \def\B{\scr B} \def\i{{\rm i}} \def\HR{\hat{\R}^d}
\def\to{\rightarrow}\def\l{\ell}\def\iint{\int}
\def\EE{\scr E}\def\no{\nonumber}
\def\A{\scr A}\def\V{\mathbb V}\def\osc{{\rm osc}}
\def\BB{\scr B}\def\Ent{{\rm Ent}}\def\3{\triangle}\def\H{\scr H}
\def\U{\scr U}\def\8{\infty}\def\1{\lesssim}\def\HH{\mathrm{H}}
 \def\T{\scr T}
\maketitle

\begin{abstract} In this paper, the Harnack inequalities for $G$-SDEs with degenerate noise are derived by method of coupling by change of measure. Moreover, the gradient estimate for the associated nonlinear semigroup $\bar{P}_t$
$$|\nabla \bar{P}_t f|\leq c(p,t)(\bar{P}_t |f|^p)^{\frac{1}{p}}, \ \ p>1, t>0$$
is also obtained for bounded and continuous function $f$. As an application of Harnack inequality, we prove the weak existence of degenerate $G$-SDEs under some integrable conditions. Finally, an example is presented. All of the above results extends the existed results in the linear expectation setting.
\end{abstract} \noindent
 \vskip 2cm

\section{Introduction}
Since Peng \cite{peng2, peng1, peng4}  established the fundamental theory  of $G$-Brownian motion and  SDEs driven by it ($G$-SDEs, in short), the study of  $G$-expectation has received much attention, see a summary paper \cite{P}  and references within for details.
The $G$-expectation has been applied in many areas, for instance, stochastic optimization \cite{HJ, HJ2}, financial markets with volatility uncertainty \cite{Denis} and the Feyman-Kac formula \cite{HJPS}.

Recently, using method of coupling by change of measure introduced by Wang \cite[Chapter 1]{Wbook}, Song \cite{song} studied the gradient estimates for nonlinear diffusion semigroups, where the noise is  assumed to be  non-degenerate.
Quite recently, under the nonlinear expectation framework, Yang \cite{Yang} obtained the dimensional-free Harnack inequality for $G$-SDEs with non-degenerate noise.

On the other hand, the stochastic Hamiltonian system in the linear probability space, a typical model of degenerate diffusion system, has been investigated in \cite{GW,W2,WZ1}.

In this paper, we intend to investigate Harnack inequalities and gradient estimate for $G$-SDEs with degenerate noise, i.e. the stochastic Hamiltonian system driven by $G$-Brownian motion. The method is also coupling by change of measure, in which the Girsanov transform in \cite{HJPS} palys a crucial role. Due to the lack of additivity of nonlinear expectation, the Bismut formula \cite[(1.8), (1.14)]{Wbook}, which is an important technique to get gradient estimate, can not be proved either by coupling by change of measure or Malliavin calculus in the $G$-SDEs. Instead, we directly estimate the local Lipschitz constant defined below. Moreover, as an application of Harnack inequality, we will prove the existence of weak solution for degenerate $G$-SDEs perturbed by a  drift which only satisfies some integrable condition with respect to a reference nonlinear expectation.

Since the quadratic variation process of $G$-Brownian motion is a stochastic process, the $G$-SDEs generally contain two drift terms: $\d t$ and quadratic variation process. Moreover, the Girsanov transform is different from the one in the linear expectation case, see Theorem \ref{lem2} below for more details.
\section{Preparations}
\subsection{$G$-Expectation and $G$-Brownian motion}
Before moving on, we recall some basic facts on $G$-expectation and $G$-Brownian motion. Let $\Omega=C_0([0,\infty);\mathbb{R}^d)$, the $\mathbb{R}^d$-valued and continuous functions on $[0,\infty)$ vanishing at zero, equipped with metric
$$\rho(\omega^1,\omega^2)=\sum_{n=1}^\infty\frac{1}{2^n}\left[\max_{t\in[0,n]}|\omega^1_t-\omega^2_t|\wedge1 \right],\ \ \omega^1,\omega^2\in\Omega.$$
For any $T>0$, set
 $$L_{ip}(\Omega_T) =\{\omega\to \varphi(\omega_{t_1}, \cdot \cdot \cdot, \omega_{t_n}):n\in \mathbb{N}^+, t_1,\cdot \cdot \cdot, t_n\in [0,T],\varphi \in C_{b,lip}(\mathbb{R}^{d}\otimes \R^n)\},$$
 and $$L_{ip}(\Omega)=\bigcup_{T>0}L_{ip}(\Omega_T),$$
where  $C_{b,lip}(\mathbb{R}^{d}\otimes\R^n)$ denotes  the set of bounded and Lipschitz continuous functions on $\mathbb{R}^{d}\otimes \R^n$.
Let $\mathbb{S}^d$ be the collection of all $d\times d$ symmetric matrices and $\mathbb{S}_+^d\subset \mathbb{S}^d$ denote all $d\times d$ positive definite and symmetric matrices.
Fix $\underline{\sigma}, \overline{\sigma}\in \mathbb{S}_+^d$ with $\underline{\sigma}<\overline{\sigma}$ and define
\begin{equation}\label{G(A)}
 G(A):=\frac{1}{2}\sup _{\gamma\in [\underline{\sigma}, \bar{\sigma}]}\mbox{trace}(\gamma^2 A), \ A\in\mathbb{S}^d.
\end{equation}
Then it is not difficult to see
  \begin{equation} \label{Gnon}
     G(A)-G(\bar{A})\geq \frac{\lambda_0(\underline{\sigma}^2)}{2} \mbox{trace}[A-\bar{A}],\ A\geq \bar{A}, A,\bar{A} \in \mathbb{S}^d,
  \end{equation}
where $\lambda_0(\underline{\sigma}^2)>0$ is the minimal eigenvalue of $\underline{\sigma}^2$.

 Let $\bar{\E}^G$ be the nonlinear expectation on $\Omega$ such that coordinate process $B=(B_t)_{t\geq 0}$, i.e. $B_t(\omega)=\omega_t, \omega\in \Omega$, is a $d$-dimensional $G$-Brownian motion on $(\Omega, L_G^1(\Omega),\bar{\E}^G)$, where $L_G^1(\Omega)$ is the completion of $L_{ip}(\Omega)$  under the norm $(\bar{\E}^G|\cdot|)$. See \cite{song} for details on the construction of $\bar{\E}^G$. For any $p\geq 1$, let $L_G^p(\Omega)$ be the completion of $L_{ip}(\Omega)$  under the norm $(\bar{\E}^G|\cdot|^p)^{\frac{1}{p}}$. Similarly, we can define $L_G^p(\Omega_T)$ for any $T>0$.

Let
\begin{equation*}\label{equa11} M_{G}^{p,0}([0,T])=
\Big\{\eta_t:=\sum_{j=0}^{N-1} \xi_{j} I_{[t_j, t_{j+1})}(t);
~\xi_{j}\in L_{G}^p(\Omega_{t_{j}}),
 N\in\mathbb{N}^+,\ 0=t_0<t_1<\cdots <t_N=T \Big\},
\end{equation*}
and $M_G^p([0,T])$ be the completion of $M_G^{p,0}([0,T])$ under the norm
$$\|\eta\|_{M_G^p([0,T])}:=\left(\bar{\E}^G\int_{0}^{T}|\eta_{t}|^p\d t\right)^{\frac{1}{p}}.$$
Moreover, let $$ M_G^2([0,T])^d=\left\{X=(X^1,X^2,\cdots,X^d), X^i\in M_G^2([0,T]),1\leq i\leq d\right\}.$$
Let $\mathcal{M}$ be the collection of all probability measures on  $(\Omega, \mathcal{B}(\Omega))$.
According to \cite{15,HP},    there exists a weakly compact subset $\mathcal{P}\subset \mathcal{M}$   such that
\begin{align}\label{rep}\bar{\E}^G[X]=\sup_{P\in \mathcal{P}}\mathbb{E}_P[X], \ X\in L_G^1(\Omega),\end{align}
where $\mathbb{E}_P$ is the linear expectation under probability measure $P \in \mathcal{P}$. $\mathcal{P}$ is called a set that represents $\bar{\E}^G$.
In fact, let $W^0$ be a $d$-dimensional Brownian motion on a complete filtration probability space $(\OO,\{\F_t\}_{t\geq 0},\P)$, and $\mathbb{H}$ be the set of all progressively measurable stochastic processes valued in $[\underline{\sigma}, \bar{\sigma}]$. For any $\theta\in \mathbb{H}$, define $\P_{\theta}$ as the law of $\int_0^\cdot \theta_s \d W^0_s$. Then by \cite{15,HP}, we can take $\mathcal{P}= \{\P_\theta, \theta\in\mathbb{H}\}$, i.e.
\begin{align}\label{rep2}
\bar{\E}^G[X]=\sup_{\theta\in \mathbb{H}}\mathbb{E}_{\P_\theta}[X], \ X\in L_G^1(\Omega).
\end{align}
The associated  Choquet capacity to $\bar{\E}^G$ is defined by
$$\mathcal{C}(A)=\sup_{P\in \mathcal{P}}P(A), \ A\in \mathcal{B}(\Omega).$$
A set $A\subset \Omega$ is called polar if $\mathcal{C}(A)=0$,  and we say that a property   holds  $\mathcal{C}$-quasi-surely ($\mathcal{C}$-q.s.)
if it holds outside a polar set, see \cite{15} for more details on capacity.

Finally, letting $\<B\>$ be the quadratic variation process of $B$, then by \eqref{Gnon} and \cite[Chapter III, Corollary 5.7]{peng4}, we have $\mathcal{C}$-q.s.
\begin{align}\label{B}
\underline{\sigma}^2<\frac{\d}{\d t}\langle B\rangle_t\leq\bar{\sigma}^2.
\end{align}
\subsection{Girsanov's Transform}
The following Girsanov's  transform comes from \cite[Proposition 5.10]{O}.
\begin{thm} \label{Gir}
Let $\{g_t\}_{t\leq T}\in M_G^2([0,T])^d$. If there exists a constant $\delta>0$ such that
$$\bar{\E}^G\exp\left\{\left(\frac{1}{2}+\delta\right)\int_0^T\<g_u,\d \langle B\rangle_ug_u\>\right\}<\infty.$$
Then
$$
\bar{B}:=B+\int_{0}^{\cdot}\d \langle B\rangle_ug_u
$$
is a $G$-Brownian motion on $[0,T]$ under $\tilde{\E}[\cdot]=\bar{\E}^{G}[\tilde{R}_T(\cdot)]$, where
\begin{align*}
\tilde{R}_T&=\exp\bigg[-\int_0^T\left\< g_u, \d
B_u\right\>-\frac{1}{2}\int_0^T\<g_u,\d \langle B\rangle_ug_u\>\bigg].
\end{align*}
\end{thm}
According to \cite[Remark 5.3]{HJPS},
letting $\hat{\Omega}=C_0([0,\infty),\mathbb{R}^{2d})$,  we can construct an auxiliary $\hat{G}$-expectation space $(\hat{\Omega}, L_{\hat{G}}^1(\hat{\Omega}),\hat{\E}^{\hat{G}})$ with
$$\hat{G}(A):=\frac{1}{2}\sup _{\gamma\in [\underline{\sigma}, \bar{\sigma}]}\mathrm{trace}\left[A
\left(
  \begin{array}{cc}
    \gamma^2 & 1 \\
    1 & \gamma^{-2} \\
  \end{array}
\right)\right],\ A\in \mathbb{S}^{2d},$$
and a $d$-dimensional process $B'$ such that $\left(
                                                \begin{array}{c}
                                                  B \\
                                                  B' \\
                                                \end{array}
                                              \right)
$ is a $2d$-dimensional $\hat{G}$-Brownian motion and $\left\<B,B'\right\>_t=tI_{d\times d}$ under $\hat{\E}^{\hat{G}}$. In addition, 
the distribution of $B$ under $\bar{\E}^G$ is equal to that of $B$ under $\hat{\E}^{\hat{G}}$. Moreover, letting
\begin{equation}\label{G'}
 \tilde{G}(A)=\frac{1}{2}\sup _{\gamma\in [\underline{\sigma}, \bar{\sigma}]}\mathrm{trace}\left[A\gamma^{-2}\right],\ \ A\in \mathbb{S}^{d},
\end{equation}
  then $B'$ is a $\tilde{G}$-Brownian motion under $\hat{\E}^{\hat{G}}$. Letting $\hat{\mathcal{C}}$ be the associated Choquet capacity to $\hat{\E}^{\hat{G}}$, then we have $\hat{\mathcal{C}}$-q.s.
\begin{align}\label{B'}
\overline{\sigma}^{-2}\leq \frac{\d \langle B'\rangle_t}{\d t}\leq \underline{\sigma}^{-2}.
\end{align}
As a corollary of Theorem \ref{Gir}, we have the following Girsanov's transform, which will be used in the sequel.
\begin{thm} \label{lem2}
Let $\{g^i_t\}_{t\leq T}\in M_G^2([0,T])^d, i=1,2$. If
\begin{align}\label{Nov}
&\hat{\E}^{\hat{G}}\exp\bigg\{\left(\frac{1}{2}+\delta\right)\int_0^T \left(\<g^1_s,\d \langle B'\rangle_sg^1_s\>+\<g^2_s,\d \langle B\rangle_sg^2_s\>+2\<g^1_s,g^2_s\>\d s\right)\bigg\}<\infty.
\end{align}
Then
$$
\breve{B}:=B+\int_{0}^{\cdot}g^1_u\d u+\int_{0}^{\cdot}g^2_u\d \langle B\rangle_u
$$
is a $G$-Brownian motion on $[0,T]$ under $\breve{\E}[\cdot]=\hat{\E}^{\hat{G}}[\breve{R}_T(\cdot)]$, where
\begin{align*}
\breve{R}_T&=\exp\bigg[-\int_0^T\left\< \left(
             \begin{array}{c}
               g^1_u  \\
                 g^2_u\\
             \end{array}
           \right),
           \d \left(
             \begin{array}{c}
               B'_u,\\
               B_u  \\
             \end{array}
           \right)
\right\>\\
&\qquad\qquad-\frac{1}{2}\int_0^T \left(\<g^1_s,\d \langle B'\rangle_sg^1_s\>+\<g^2_s,\d \langle B\rangle_sg^2_s\>+2\<g^1_s,g^2_s\>\d s\right)\bigg].
\end{align*}
\end{thm}
\begin{proof}
Letting $W=\left(
             \begin{array}{c}
               B \\
               B' \\
             \end{array}
           \right),$
we have
$$\<W\>_t=\left(
          \begin{array}{cc}
            \<B\>_t & tI_{d\times d} \\
            tI_{d\times d} & \<B'\>_t \\
          \end{array}
        \right)
$$
and \begin{align*}&\int_0^T\left\<\left(
                                    \begin{array}{c}
                                      g^2 \\
                                      g^1 \\
                                    \end{array}
                                  \right),\d \<W\>\left(
                                    \begin{array}{c}
                                      g^2 \\
                                      g^1 \\
                                    \end{array}
                                  \right)\right\>\\
                                  &=\int_0^T \left(\<g^1_s,\d \langle B'\rangle_sg^1_s\>+\<g^2_s,\d \langle B\rangle_sg^2_s\>+2\<g^1_s,g^2_s\>\d s\right).
\end{align*}
Let
$$\tilde{W}=W+\int_0^\cdot\d \<W\>\left(
                                    \begin{array}{c}
                                      g^2 \\
                                      g^1 \\
                                    \end{array}
                                  \right)
.$$
In view of \eqref{Nov} and applying Theorem \ref{Gir} for $\left(W,\left(
                                    \begin{array}{c}
                                      g^2 \\
                                      g^1 \\
                                    \end{array}
                                  \right)\right)$ replacing $(B,g)$, we conclude that $\tilde{W}$ is a $2d$-dimensional $\hat{G}$-Brownian motion on $[0,T]$ under $\breve{\E}$ defined in Theorem \ref{lem2}. In particular,
$$\breve{B}:=B+\int_{0}^{\cdot}g^1_u\d u+\int_{0}^{\cdot}g^2_u\d \langle B\rangle_u$$
is a $G$-Brownian motion on $[0,T]$ under $\breve{\E}$.
\end{proof}
\begin{rem}\label{Gi} Theorem \ref{lem2} extends the result in \cite[Theorem 5.2]{HJPS}, where $g^1$ and $g^2$ are assumed to be bounded processes.
\end{rem}
Throughout the paper, the letter $C$ or $c$ will denote a positive constant, and $C(\theta)$ or $c(\theta)$ stands for a constant depending on $\theta$. The value of the constants may change from one appearance to another.

\section{Harnack and Gradient Estimate}
Consider the following $G$-SDE on $\mathbb{R}^{m+d}$:
\beq\label{EH}
\begin{cases}
\d X_t=\{AX_t+MY_t\}\d t, \\
\d Y_t=b_1(X_t,Y_t)\d t+\d \<B\>_tb_2(X_t,Y_t)+Q\d B_t,
\end{cases}
\end{equation}
where $B_t$ is a $d$-dimensional $G$-Brownian motion defined in Section 1, $A$ is an $m\times m$ matrix, $M$ is an $m\times d$ matrix, $Q$ is a $d\times d$ matrix, $b_1, b_2:\mathbb{R}^{m+d}\to \mathbb{R}^d$.

 In this paper, we only consider $m=d=1$, and the result can be extended to general $m\geq 1$ and $d\geq 1$. In this  case, $\underline{\sigma}$ and $\overline{\sigma}$ in \eqref{G(A)} are two positive constants satisfying $\underline{\sigma}< \overline{\sigma}$, and the corresponding generating function is given by
 $$G(a)=\frac{1}{2}\overline{\sigma}^2a^+-\frac{1}{2}\underline{\sigma}^2 a^-,\ \ a\in \mathbb{R}^1.$$
In this section,  we study the Harnack inequalities and gradient estimate for \eqref{EH}. To this end,  we make the following assumptions:
\begin{enumerate}
\item[\bf{(A1)}]  $QM\neq 0$.
\item[\bf{(A2)}]  There exists $K>0$ such that
\beq\label{Z}\beg{split}
|b_1(z)-b_1(\bar{z})|+|b_2(z)-b_2(\bar{z})|\leq K|z-\bar{z}|,\ \ z,\bar{z}\in\mathbb{R}^{2}.
\end{split}\end{equation}
\end{enumerate}
\begin{rem}\label{non-ex}
According to \cite[Theorem 1.2]{peng4}, {\bf(A2)} implies that \eqref{EH} has a unique non-explosive strong solution $(X^z_t,Y^z_t)$ in $M_G^2([0,T])^2$ for any $T>0$ and  $(X_0$, $Y_0)=z\in\mathbb{R}^{2}$.
\end{rem}
Denote by $C_b(\mathbb{R}^2)$ ($C^+_b(\mathbb{R}^2)$) the bounded (non-negative bounded) and continuous function on $\mathbb{R}^2$. Let $\bar{P}_t$ be the associated nonlinear  semigroup to $(X^z_t,Y^z_t)$, i.e.
$$\bar{P}_tf(z)=\bar{\E}^Gf(X^z_t,Y^z_t), \ \ f\in C_b(\mathbb{R}^{2}).$$
For a real-valued functtion $f$ defined on a metric sapce $(H, \rho)$, define
\begin{align}\label{mod}
|\nabla f(z)|=\limsup_{\rho(x,z)\to 0}\frac{|f(x)-f(z)|}{\rho(x,z)}, \ \ z\in H.
\end{align}
Then $|\nabla f(z)|$ is called the local Lipschitz constant of $f$ at point $z\in H$.
\begin{thm}\label{T3.2} Assume
  {\bf (A1)}-{\bf (A2)} and let $T>0$. Then there exists some constant $C>0$ depending on $A$, $K$ and $|Q^{-1}|$ such that the following assertions hold.
\begin{enumerate}
\item[(1)] For any $z=(z_1, z_2), h=(h_1, h_2)\in \mathbb{R}^{2}$, $p>1$, the Harnack inequality
 \beg{equation}\label{H}
 \beg{split}
 (\bar{P}_Tf)^p(z+h)\le  &\bar{P}_Tf^p(z)
  \exp\bigg[C\ff{p}{2(p-1)} \Sigma(T)|h|^2\bigg],\ \ f\in C^+_b(\mathbb{R}^2)
  \end{split}\end{equation}
holds with 
\begin{align}\label{sigma}
\Sigma(T):=\underline{\sigma}^{-2}T\left(\ff 1 {T}+ \ff{1}{T^{2}}+1+ T\right)^2+\overline{\sigma}^{2}T\left(1+ T\right)^2.
\end{align}
\item[(2)]
The gradient estimate, i.e.
\begin{align}\label{ge}\|\nabla\bar{P}_T f\|_{\infty}\leq C\|f\|_\infty\sqrt{\Sigma(T)},\ \ f\in C^+_b(\mathbb{R}^2)
\end{align}
holds for some constant $C>0$.
\item[(3)] For any $p>1$, there exists a constant $c(p)>0$ such that
\begin{align}\label{gep}
|\nabla \bar{P}_T f(z)|
&\leq c(p)\left(\bar{P}_T|f|^p(z)\right)^{\frac{1}{p}}\sqrt{\Sigma(T)},\ \ z\in\mathbb{R}^2, \ \ f\in C^+_b(\mathbb{R}^2).
\end{align}
\end{enumerate}
 \end{thm}
 \begin{rem} With \eqref{rep} in hand, it seems that \eqref{H} can be derived by taking a supremum in the following Harnack inequality for the linear expectation $\E_P$:
 \begin{align} \label{lin}
 \left(\E_P f(X^{z+h}_t, Y^{z+h}_t)\right)^p\leq \left(\E_P f^p(X^{z}_t, Y^{z}_t)\right)\exp\{\Phi(t,h,p)\}.
 \end{align}
However, the method of coupling by change of measure is available for the  SDEs driven by Brownian motion. So, it is difficult to get \eqref{lin} since $B_t$ is only a martingale under $\E_P$. Therefore, the results in Theorem \ref{T3.2} are non-trivial.
 \end{rem}
 \begin{rem} Compared to the SDEs in \cite{song}, the SDE \eqref{EH} is allowed to contain drift $\d \<B\>_tb_2.$
 \end{rem}
Now, we are in the position to prove Theorem \ref{T3.2}.
\begin{proof}
\begin{enumerate}
  \item [(1)]
For any $\eta\in\mathbb{R}^{2}$, let $(X^\eta_t, Y^\eta_t)$ solve (\ref{EH})  with $(X_0,Y_0)= \eta$. For $h=(h_1,h_2)\in\mathbb{R}^{2}$, define
$$\gamma_1(s)= v_1(s) h_2+\alpha_1(s), \ \ s\in[0,T]$$
with
\begin{align*}
&v_1(s)= \ff{T-s}{T},\\
&\alpha_1(s)= -\ff{s(T-s)}{T^2}M\e^{-sA}\Lambda_1(T)^{-1} \bigg(h_1+\int_0^{T} \ff{T-u}{T} \e^{-uA}Mh_2\d u\bigg),\ \ s\in[0,T],
\end{align*}
where
$$\Lambda_1(T):= \int_0^T \ff{s(T-s)}{T^2}\e^{-2sA} M^2\d s.$$
It is clear that
\beq\label{QQ0} |\Lambda_1(T)^{-1}|\le cT^{-1}\end{equation} holds for some constant $c>0.$

Let $(\tilde{X}_t, \tilde{Y}_t)$ solve the equation
\beq\label{EC10} \beg{cases} \d\tilde{X}_t= \{A\tilde{X}_t+M\tilde{Y}_t\}\d t,\\
\d \tilde{Y}_t= b_1(X^z_t, Y^z_t)\d t +b_2(X^z_t, Y^z_t)\d \<B\>_t+Q\d B_t+ \gamma_1'(t)\d t\end{cases}\end{equation} with
$(\tilde{X}_0,\tilde{Y}_0)= z+h$.
Then the solution to (\ref{EC10}) is non-explosive as well.

Set
 $$\Theta_1(s)=\left(\e^{As}h_1+\int_0^{s} \e^{(s-u)A} M\gamma_1(u)\d u,\ \gamma_1(s)\right),\ \ s\in[0,T].$$
Then there exists a constant $C>0$ such that for any $s\in[0,T]$,
\beg{equation}\label{NN0}\beg{split}
&|\gamma_1'(s)|\le C\left(\ff 1 {T}+ \ff{1}{T^{2}}\right)|h|,\\
&|\Theta_1(s)|\le C\Big(1+ T\Big)|h|.\end{split}\end{equation}
Note that
\beq\label{EE} (\tilde{X}_s,\tilde{Y}_s)=(X^z_s,Y^z_s)+\Theta_1(s), \ \
s\in[0,T],\end{equation}
and in particular, $(\tilde{X}_T,\tilde{Y}_T)=(X^z_T,Y^z_T)$.
Let
\begin{align*}
&\Phi_1(s)=Q^{-1}\{b_1(X^z_s, Y^z_s)-b_1(\tilde{X}_s,
\tilde{Y}_s)+\gamma_1'(s)\},\\
&\Phi_2(s)=Q^{-1}\{b_2(X^z_s, Y^z_s)-b_2(\tilde{X}_s,
\tilde{Y}_s)\},\ \ s\in[0,T],
\end{align*}
and $B'$ be in Section 2.2.
\eqref{B},  \eqref{B'},  \eqref{NN0} and \eqref{EE} together with {\bf(A1)}-{\bf(A2)}  imply $\hat{\mathcal{C}}$-q.s.
\begin{equation}\begin{split}\label{Phi}
&\int_0^T|\Phi_1(s)|^2\d \langle B'\rangle_s+\int_0^T|\Phi_2(s)|^2\d \langle B\rangle_s+2\int_0^T\Phi_1(s)\Phi_2(s)\d s\\
&\leq \int_0^T\underline{\sigma}^{-2}|\Phi_1(s)|^2\d s+\int_0^T\overline{\sigma}^2|\Phi_2(s)|^2\d s+2\int_0^T\Phi_1(s)\Phi_2(s)\d s\\
&\leq 2\int_0^T\underline{\sigma}^{-2}|\Phi_1(s)|^2\d s+2\int_0^T\overline{\sigma}^2|\Phi_2(s)|^2\d s\\
&\leq C\underline{\sigma}^{-2}\int_0^T\left(|\Theta_1(s)|+|\gamma_1'(s)|\right)^2\d s+C\overline{\sigma}^{2}\int_0^T|\Theta_1(s)|^2\d s\\
&\leq C\underline{\sigma}^{-2}T\left(\ff 1 {T}+ \ff{1}{T^{2}}+1+ T\right)^2|h|^2+C\overline{\sigma}^{2}T\left(1+ T\right)^2|h|^2\\
&=C\Sigma(T)|h|^2
\end{split}\end{equation}
for some constant $C>0$ depending on $A$, $K$ in {\bf(A2)} and $|Q^{-1}|$.

Applying Theorem \ref{lem2},  we conclude that
$$
\tilde{B}:=B+\int_{0}^{\cdot}\Phi_1(u)\d u+\int_{0}^{\cdot}\Phi_2(u)\d \<B\>_u
$$
is a $G$-Brownian motion on $[0,T]$ under $\E_1(\cdot)=\hat{\E}^{\hat{G}}(R_1(T)(\cdot))$, where
\begin{align*}
R_1(T)&=\exp\bigg[-\int_0^T\left\< \left(
                                      \begin{array}{c}
                                        \Phi_1(u)\\
                                        \Phi_2(u) \\
                                      \end{array}
                                    \right)
, \d \left(
                                      \begin{array}{c}
                                        B'_u\\
                                        B_u \\
                                      \end{array}
                                    \right)\right\>\\
&\qquad\qquad-\frac{1}{2}\int_0^T \left(|\Phi_1(s)|^2\d \langle B'\rangle_s+|\Phi_2(s)|^2\d \langle B\rangle_s+2\Phi_1(s)\Phi_2(s)\d s\right)\bigg].
\end{align*}
Since $\<\tilde{B}\>=\<B\>$, (\ref{EC10}) reduces to
\beq\label{E2'0}   \beg{cases} \d \tilde{X}_t= \{A\tilde{X}_t+M\tilde{Y}_t\}\d t,\\
\d \tilde{Y}_t= b_1(\tilde{X}_t, \tilde{Y}_t)\d t+\d \<\tilde{B}\>_tb_2(\tilde{X}_t, \tilde{Y}_t) +Q\d \tilde{B}_t.
\end{cases}
\end{equation}
This means that the distribution of $(\tilde{X}_t,\tilde{Y}_t)$ under $\E_1$ coincides with that of $(X^{z+h}_t$, $Y^{z+h}_t)$  under $\hat{\E}^{\hat{G}}$ (or $\bar{\E}^G$).
Thus, H\"{o}lder inequality implies for any $f\in C_b^+(\mathbb{R}^2)$ and $p>1$,
\begin{align}\label{couple}
\bar{P}_T f(z+h)&=\E_1f(\tilde{X}_T,\tilde{Y}_T)\nonumber\\
&=\hat{\E}^{\hat{G}}\left[R_1(T)f(X^z_T,Y^z_T)\right]\\
&\leq (\bar{P}_T f^p(z))^{\frac{1}{p}}\{\hat{\E}^{\hat{G}} R_1(T)^{\frac{p}{p-1}}\}^{\frac{p-1}{p}}\nonumber,
\end{align}
here we used the fact that the distribution of $B$ under $\bar{\E}^G$ is equal to that of $B$ under $\hat{\E}^{\hat{G}}$.
It follows from the definition of $R_1(T)$ and  \eqref{Phi} that
\begin{equation*}\begin{split}\label{exp}
&\hat{\E}^{\hat{G}} R_1(T)^{\frac{p}{p-1}}\\
&= \hat{\E}^{\hat{G}}\Bigg\{\exp\bigg[-\frac{p}{p-1}\int_0^T\left\< \left(
                                      \begin{array}{c}
                                        \Phi_1(u)\\
                                        \Phi_2(u) \\
                                      \end{array}
                                    \right)
, \d \left(
                                      \begin{array}{c}
                                        B'_u\\
                                        B_u \\
                                      \end{array}
                                    \right)\right\>\\
&\qquad-\frac{1}{2}\frac{p^2}{(p-1)^2}\int_0^T \left(|\Phi_1(s)|^2\d \langle B'\rangle_s+|\Phi_2(s)|^2\d \langle B\rangle_s+2\Phi_1(s)\Phi_2(s)\d s\right)\bigg]\\
&\qquad\times\exp\bigg[\frac{p}{2(p-1)^2}\int_0^T \left(|\Phi_1(s)|^2\d \langle B'\rangle_s+|\Phi_2(s)|^2\d \langle B\rangle_s+2\Phi_1(s)\Phi_2(s)\d s\right)\Bigg\}\\
&\leq \exp\bigg[C\frac{p}{2(p-1)^2}\Sigma(T)|h|^2\bigg].
\end{split}\end{equation*}
Combining this with \eqref{couple}, we derive the Harnack inequality \eqref{H}.
 \item[(2)]
Now, we prove the gradient estimate \eqref{ge}. Since the distribution of $B$ under $\bar{\E}^G$ is equal to that of $B$ under $\hat{\E}^{\hat{G}}$, we have
$$\bar{P}_T f(z)=\bar{\E}^Gf(X^z_T,Y^z_T)=\hat{\E}^{\hat{G}}f(X^z_T,Y^z_T).$$
This and \eqref{couple} yield
\begin{equation}\begin{split}\label{PP}
|\bar{P}_T f(z+h)-\bar{P}_T f(z)|&=|\hat{\E}^{\hat{G}}\left[R_1(T)f(X^z_T,Y^z_T)\right]-\hat{\E}^{\hat{G}}f(X^z_T,Y^z_T)|\\
&\leq \hat{\E}^{\hat{G}}\left(|f(X^z_T,Y^z_T)||R_1(T)-1|\right).
\end{split}\end{equation}
Noting that $|x-1|\leq (x+1)|\log x|$ for any $x>0$, we have
\begin{equation}\begin{split}\label{pp0}
&|\bar{P}_T f(z+h)-\bar{P}_T f(z)|\\
&\leq \|f\|_{\infty}\hat{\E}^{\hat{G}}R_1(T)|\log R_1(T)|+\|f\|_{\infty}\hat{\E}^{\hat{G}}|\log R_1(T)|\\
&= \|f\|_{\infty}\left(\E_1|\log R_1(T)|+\hat{\E}^{\hat{G}}|\log R_1(T)|\right).
\end{split}\end{equation}
Let
$$
\tilde{B'}=B'+\int_{0}^{\cdot}\Phi_1(u)\d \langle B'\rangle_u+\int_{0}^{\cdot}\Phi_2(u)\d u.
$$
From Theorem \ref{lem2}, we know that $\tilde{B'}$ is a $\tilde{G}$-Brownian motion under $\E_1$.
Applying B-D-G inequality and noting $\<\tilde{B'}\>=\<B'\>$ and $\<\tilde{B}\>=\<B\>$, we obtain
\begin{align*}
&\E_1|\log R_1(T)|\\
&=\E_1\bigg|-\int_0^T\left\< \left(
                                      \begin{array}{c}
                                        \Phi_1(u)\\
                                        \Phi_2(u) \\
                                      \end{array}
                                    \right)
, \d \left(
                                      \begin{array}{c}
                                        B'_u\\
                                        B_u \\
                                      \end{array}
                                    \right)\right\>\\
&-\frac{1}{2}\int_0^T \left(|\Phi_1(s)|^2\d \langle B'\rangle_s+|\Phi_2(s)|^2\d \langle B\rangle_s+2\Phi_1(s)\Phi_2(s)\d s\right)\bigg|\\
&\leq \E_1\Bigg|\int_0^T\left\< \left(
                                      \begin{array}{c}
                                        \Phi_1(u)\\
                                        \Phi_2(u) \\
                                      \end{array}
                                    \right)
, \d \left(
                                      \begin{array}{c}
                                        \tilde{B}'_u\\
                                        \tilde{B}_u \\
                                      \end{array}
                                    \right)\right\>\Bigg|\\
&+\frac{1}{2}\E_1\Bigg|\int_0^T \left(|\Phi_1(s)|^2\d \langle \tilde{B'}\rangle_s+|\Phi_2(s)|^2\d \langle \tilde{B}\rangle_s+2\Phi_1(s)\Phi_2(s)\d s\right)\Bigg|\\
&\leq \E_1\left(\int_0^T \left(|\Phi_1(s)|^2\d \langle \tilde{B'}\rangle_s+|\Phi_2(s)|^2\d \langle \tilde{B}\rangle_s+2\Phi_1(s)\Phi_2(s)\d s\right)\right)^{\frac{1}{2}}\\
&+\frac{1}{2}\E_1\Bigg|\int_0^T \left(|\Phi_1(s)|^2\d \langle \tilde{B'}\rangle_s+|\Phi_2(s)|^2\d \langle \tilde{B}\rangle_s+2\Phi_1(s)\Phi_2(s)\d s\right)\Bigg|.
\end{align*}
This together with \eqref{Phi} implies
\begin{align*}
\E_1|\log R_1(T)|&\leq C\left(\Sigma(T)|h|^2+\sqrt{\Sigma(T)}|h|\right),
\end{align*}
here, $\Sigma(T)$ is defined in \eqref{sigma}.
Similarly, we get
\begin{align*}
\hat{\E}^{\hat{G}}|\log R_1(T)|&\leq C\left(\Sigma(T)|h|^2+\sqrt{\Sigma(T)}|h|\right).
\end{align*}
Then it follows from \eqref{pp0} that
\begin{align}\label{ge0}|\bar{P}_T f(z+h)-\bar{P}_T f(z)|\leq C\|f\|_{\infty}\left(\Sigma(T)|h|^2+\sqrt{\Sigma(T)}|h|\right).
\end{align}
This together with \eqref{mod} yields
\begin{align}\label{ge1}|\nabla\bar{P}_T f(z)|\leq C\|f\|_\infty\sqrt{\Sigma(T)},
\end{align}
which implies \eqref{ge}.
\item[(3)]
In order to get \eqref{gep}, let
\begin{align*}\tilde{R}_1(T)&=\exp\bigg[-\frac{p}{p-1}\int_0^T\left\< \left(
                                      \begin{array}{c}
                                        \Phi_1(u)\\
                                        \Phi_2(u) \\
                                      \end{array}
                                    \right)
, \d \left(
                                      \begin{array}{c}
                                        B'_u\\
                                        B_u \\
                                      \end{array}
                                    \right)\right\>\\
&\quad-\frac{1}{2}\frac{p^2}{(p-1)^2} \int_0^T \left(|\Phi_1(s)|^2\d \langle B'\rangle_s+|\Phi_2(s)|^2\d \langle B\rangle_s+2\Phi_1(s)\Phi_2(s)\d s\right)\bigg],
\end{align*}
and
\begin{equation}\begin{split}\label{BB}
&\hat{B'}=B'+\int_{0}^{\cdot}\frac{p}{p-1}\Phi_1(u)\d \langle B'\rangle_u+\int_{0}^{\cdot}\frac{p}{p-1}\Phi_2(u)\d u,\\
&\hat{B}=B+\int_{0}^{\cdot}\frac{p}{p-1}\Phi_1(u)\d u+\int_{0}^{\cdot}\frac{p}{p-1}\Phi_2(u)\d \langle B\rangle_u.
\end{split}\end{equation}
Again by Theorem \ref{lem2}, $\hat{B'}$ is also a $\tilde{G}$-Brownian motion under
$\tilde{\E}_2(\cdot)=\hat{\E}^{\hat{G}}(\tilde{R}_1(T)(\cdot))$.
Using the inequality $|x-1|\leq (x+1)|\log x|$ for any $x>0$ and \eqref{Phi}, we have
\begin{equation}\begin{split}\label{RR}
&\hat{\E}^{\hat{G}}|R_1(T)-1|^{\frac{p}{p-1}}\\
&\leq \hat{\E}^{\hat{G}}|R_1(T)+1|^{\frac{p}{p-1}}|\log R_1(T)|^{\frac{p}{p-1}}\\
&\leq c(p)\hat{\E}^{\hat{G}}R_1(T)^{\frac{p}{p-1}}|\log R_1(T)|^{\frac{p}{p-1}}+c(p)\hat{\E}^{\hat{G}}|\log R_1(T)|^{\frac{p}{p-1}}\\
&\leq c(p)\exp\bigg[C\frac{p}{2(p-1)^2}\Sigma(T)|h|^2\bigg] \hat{\E}^{\hat{G}}\left(\tilde{R}_1(T)|\log R_1(T)|^{\frac{p}{p-1}}\right)\\
&\qquad+c(p)\hat{\E}^{\hat{G}}|\log R_1(T)|^{\frac{p}{p-1}}
\end{split}\end{equation}
for some constants $C, c(p)>0$.
Combining \eqref{BB} and B-D-G inequality and noting $\<\hat{B'}\>=\<B'\>$ and $\<\hat{B}\>=\<B\>$, we obtain
\begin{align*}
&\hat{\E}^{\hat{G}}\left(\tilde{R}_1(T)|\log R_1(T)|^{\frac{p}{p-1}}\right)\\
&= \tilde{\E}_2\bigg|-\int_0^T\left\< \left(
                                      \begin{array}{c}
                                        \Phi_1(u)\\
                                        \Phi_2(u) \\
                                      \end{array}
                                    \right)
, \d \left(
                                      \begin{array}{c}
                                        B'_u\\
                                        B_u \\
                                      \end{array}
                                    \right)\right\>\\
&\qquad\qquad-\frac{1}{2}\int_0^T\left(|\Phi_1(s)|^2\d \langle B'\rangle_s+|\Phi_2(s)|^2\d \langle B\rangle_s+2\Phi_1(s)\Phi_2(s)\d s\right)\bigg|^{\frac{p}{p-1}}\\
&= \tilde{\E}_2\bigg|-\int_0^T\left\< \left(
                                      \begin{array}{c}
                                        \Phi_1(u)\\
                                        \Phi_2(u) \\
                                      \end{array}
                                    \right)
, \d \left(
                                      \begin{array}{c}
                                        \hat{B'}_u\\
                                        \hat{B}_u \\
                                      \end{array}
                                    \right)\right\>\\
&\qquad+\left(\frac{p}{p-1}-\frac{1}{2}\right)\int_0^T\left(|\Phi_1(s)|^2\d \langle \hat{B'}\rangle_s+|\Phi_2(s)|^2\d \langle \hat{B}\rangle_s+2\Phi_1(s)\Phi_2(s)\d s\right)\bigg|^{\frac{p}{p-1}}\\
&\leq c(p)\tilde{\E}_2\left|\int_0^T\left\< \left(
                                      \begin{array}{c}
                                        \Phi_1(u)\\
                                        \Phi_2(u) \\
                                      \end{array}
                                    \right)
, \d \left(
                                      \begin{array}{c}
                                        \hat{B}'_u\\
                                       \hat{ B}_u \\
                                      \end{array}
                                    \right)\right\>\right|^{\frac{p}{p-1}}\\
&+c(p)\tilde{\E}_2\left|\left(\frac{p}{p-1}-\frac{1}{2}\right)\int_0^T\left(|\Phi_1(s)|^2\d \langle \hat{B'}\rangle_s+|\Phi_2(s)|^2\d \langle \hat{B}\rangle_s+2\Phi_1(s)\Phi_2(s)\d s\right)\right|^{\frac{p}{p-1}}\\
&\leq c(p)\tilde{\E}_2\left|\int_0^T\left(|\Phi_1(s)|^2\d \langle \hat{B'}\rangle_s+|\Phi_2(s)|^2\d \langle \hat{B}\rangle_s+2\Phi_1(s)\Phi_2(s)\d s\right)\right|^{\frac{p}{2(p-1)}}\\
&+c(p)\tilde{\E}_2\left|\left(\frac{p}{p-1}-\frac{1}{2}\right)\int_0^T\left(|\Phi_1(s)|^2\d \langle \hat{B'}\rangle_s+|\Phi_2(s)|^2\d \langle \hat{B}\rangle_s+2\Phi_1(s)\Phi_2(s)\d s\right)\right|^{\frac{p}{p-1}}\\
&=:I_1+I_2.
\end{align*}
Let $\tilde{\mathcal{C}}$ be the Choquet capacity associated to $\tilde{\E}_2$. Noting that $\hat{B'}$ is a $\tilde{G}$-Brownian motion under $\tilde{\E}_2$, then $\tilde{\mathcal{C}}$-q.s. \eqref{Phi} holds with $(B,B')$ replacing by $(\hat{B},\hat{B'})$. Thus, we get
$$I_1\leq c(p)\left(\Sigma(T)|h|^2\right)^{\frac{p}{2(p-1)}},$$
and
$$I_2\leq c(p)\left(\Sigma(T)|h|^2\right)^{\frac{p}{p-1}}.$$
Therefore, we have
\begin{align}\label{RR0}
&\left(\hat{\E}^{\hat{G}}\left(\tilde{R}_1(T)|\log R_1(T)|^{\frac{p}{p-1}}\right)\right)^{\frac{p-1}{p}}\leq c(p)\left(\Sigma(T)|h|^2+\sqrt{\Sigma(T)}|h|\right).
\end{align}
Similarly by B-D-G inequality and \eqref{Phi}, we arrive at
\begin{align*}
&\left(\hat{\E}^{\hat{G}}|\log R_1(T)|^{\frac{p}{p-1}}\right)^{\frac{p-1}{p}}
\leq c(p)\left(\Sigma(T)|h|^2+\sqrt{\Sigma(T)}|h|\right).
\end{align*}
This together with \eqref{PP}, \eqref{RR}, \eqref{RR0} and H\"{o}lder inequality yields
\begin{align*}&|\nabla \bar{P}_T f(z)|=\limsup_{h\to0}\frac{|\bar{P}_T f(z+h)-\bar{P}_T f(z)|}{|h|}\\
&\leq \left(\bar{P}_T|f|^p(z)\right)^{\frac{1}{p}}\limsup_{h\to0}\frac{\left(\hat{\E}^{\hat{G}}|R_1(T)-1|^\frac{p}{p-1} \right)^{\frac{p-1}{p}}}{|h|}\\
&\leq c(p)\left(\bar{P}_T|f|^p(z)\right)^{\frac{1}{p}}\sqrt{\Sigma(T)},\ \ z\in\mathbb{R}^2.
\end{align*}
This completes the proof.
\end{enumerate}
\end{proof}
\section{Applications of Harnack inequality}
As an application of Harnack inequality, in this section, we will prove the  weak existence of SDEs perturbed by an integrable  drifts with respect to an invariant nonlinear expectation of a regular $G$-SDE.  To this end, we assume that the Harnack inequality holds for the regular $G$-SDE. The main idea is to prove Novikov's condition by Harnack inequality. We should point out that the following procedure can also be applied for non-degenerate $G$-SDEs. However, to make the framework consistent, we only consider the stochastic Hamiltonian system. One can refer to \cite{WP} and \cite{WA} for the linear expectation case.
Let $A,M,Q,b_1,b_2$ and $B_t$ be introduced in Section 3 and $\bar{b}_1, \bar{b}_2:\R^{m+d}\to\R^d$.
For simplicity, we still consider $m=d=1$.
Consider the  stochastic Hamiltonian system:
 \beq\label{EH1}
\begin{cases}
\d X_t=\{AX_t+MY_t\}\d t, \\
\d Y_t=\bar{b}_1(X_t,Y_t)\d t+\bar{b}_2(X_t,Y_t)\d \<B\>_t\\
\qquad+b_1(X_t,Y_t)\d t+b_2(X_t,Y_t)\d \<B\>_t+Q\d B_t.
\end{cases}
\end{equation}
The referenced SDE is
 \beq\label{EH0}
\begin{cases}
\d X_t=\{AX_t+MY_t\}\d t, \\
\d Y_t=b_1(X_t,Y_t)\d t+b_2(X_t,Y_t)\d \<B\>_t+Q\d B_t.
\end{cases}
\end{equation}
Assume that \eqref{EH0} has a unique non-explosive strong solution $(X^z_t,Y^z_t)$ in $M_G^2([0,T])^2$ for any $T>0$ and  $(X_0$, $Y_0)=z\in\mathbb{R}^{2}$. Let $P_t^0$ be the associated nonlinear semigroup to \eqref{EH0} defined by
$$P^0_tf(z)=\bar{\E}^Gf(X^z_t,Y^z_t), \ \ f\in C_b(\mathbb{R}^{2}).$$
Before moving on, we first  introduce the definition of weak solution for the SDE \eqref{EH1}.
\begin{defn} $((\tilde{X}, \tilde{Y}),\tilde{B})$ is called a weak solution to \eqref{EH1} with initial value $(x,y)$, if $\tilde{B}$ is a $G$-Brownian motion on some nonlinear space $(\Omega, \tilde{\E})$ and
\beq\label{ws}
\begin{cases}
  \tilde{X}_s=x+\int_0^s\{A\tilde{X}_t+M\tilde{Y}_t\}\d t, \\
 \tilde{Y}_s=y+\int_0^s\bar{b}_1(\tilde{X}_t,\tilde{Y}_t)\d t+\int_0^s\bar{b}_2(\tilde{X}_t,\tilde{Y}_t)\d \<\tilde{B}\>_t\\
\qquad+\int_0^sb_1(\tilde{X}_t,\tilde{Y}_t)\d s+\int_0^sb_2(\tilde{X}_t,\tilde{Y}_t)\d \<\tilde{B}\>_t+Q\d \tilde{B}_t,\ \ s\geq 0.
\end{cases}
\end{equation}
\end{defn}
In the following, we recall the definition of invariant nonlinear expectation, see \cite{HLWZ} for more details.
\begin{defn}\label{inv}
A sublinear expectation $\E_0:C^1_{b}(\R^d)\to\R^1$ is said to be an invariant expectation of $P_t^0$, if
$$\E_0(P_t^0 f)=\E_0 f,\ \ f\in C^1_{b}(\R^d), t\geq 0.$$
\end{defn}
\begin{thm}\label{we} Assume that $P_t^0$ has a unique invariant nonlinear expectation and satisfies the Harnack inequality:
\begin{align}\label{HI}
(P_t^0|f|)^p(z)\leq (P_t^0|f|^p)(\bar{z})\e^{\Phi_p(t,z,\bar{z})},\ \ f\in C_b(\R^2), z,\bar{z}\in\R^2,t>0
\end{align}  with
\begin{align}\label{E0p}
\int_0^t\frac{\d s}{\left\{\E_0\e^{-\Phi_p(s,z,\cdot)}\right\}^{\frac{1}{p}}}<\infty, \ \ t>0, z\in\R^2
\end{align}
for each $p>1$.
If $\bar{b}_1$ and $\bar{b}_2$ are continuous and there exists a constant $\varepsilon>0$ such that $$\E_0\e^{\varepsilon(|\bar{b}_1|^2+|\bar{b}_2|^2)}<\infty.$$ Then for any $z\in\R^2$, the stochastic Hamiltonian system  \eqref{EH1} has a weak solution with initial value $z$.
\end{thm}
\begin{proof} Let $(X_t,Y_t)$ be the solution to \eqref{EH0} with initial value $z$. Define
$$\bar{B}_s=B_s-\int_0^sQ^{-1}[(\bar{b}_1(X_t,Y_t)\d t+\bar{b}_2(X_t,Y_t)\d \<B\>_t)].$$
Then \eqref{EH0} can be rewritten as
 \beq\label{EH2}
\begin{cases}
\d X_t=\{AX_t+MY_t\}\d t, \\
\d Y_t=\bar{b}_1(X_t,Y_t)\d t+\bar{b}_2(X_t,Y_t)\d \<B\>_t\\
\qquad+b_1(X_t,Y_t)\d t+b_2(X_t,Y_t)\d \<B\>_t+Q\d \bar{B}_t.
\end{cases}
\end{equation}
By the Markov property, it is sufficient to find out a constant $t_0>0$ such that
$\{\bar{B}\}_{s\in[0,t_0]}$ is a $G$-Brownian motion under
$\tilde{\E}[\cdot]=\hat{\E}^{\hat{G}}[\tilde{R}(t_0)(\cdot)]$, where
\begin{align*}
\tilde{R}(t_0)=&\exp\bigg[\int_0^{t_0}\left\< \left(
                                                \begin{array}{c}
                                                  \bar{b}_1(X_u,Y_u) \\
                                                  \bar{b}_2(X_u,Y_u) \\
                                                \end{array}
                                              \right)
, \d \left(
       \begin{array}{c}
         B'_u \\
         B_u \\
       \end{array}
     \right)\right\>\\
&\quad-\frac{1}{2}\int_0^{t_0} \left(|\bar{b}_1(X_u,Y_u)|^2\d \langle B'\rangle_u+|\bar{b}_2(X_u,Y_u)|^2\d \langle B\rangle_u+2\bar{b}_1(X_u,Y_u)\bar{b}_2(X_u,Y_u)\d u\right)\bigg].
\end{align*}
According to Theorem \ref{lem2}, we only need to prove
\begin{align*}
&\hat{\E}^{\hat{G}}\exp\bigg\{\left(\frac{1}{2}+\delta\right)\bigg(\int_0^{t_0}\left(|\bar{b}_1(X_t,Y_t)|^2\d \langle B'\rangle_t+|\bar{b}_2(X_t,Y_t)|^2\d \langle B\rangle_t+2(\bar{b}_1\bar{b}_2)(X_t,Y_t)\d t\right)\bigg)\bigg\}\\
&\leq\hat{\E}^{\hat{G}}\exp\left\{(1+2\delta)\left(\int_0^{t_0}\underline{\sigma}^{-2}|\bar{b}_1(X_t,Y_t)|^2\d t+\int_0^{t_0}\overline{\sigma}^2|\bar{b}_2(X_t,Y_t)|^2\d t\right)\right\}\\
&\leq\bar{\E}^{G}\exp\left\{(1+2\delta)(\underline{\sigma}^{-2}+\overline{\sigma}^2) \int_0^{t_0}\left(|\bar{b}_1(X_t,Y_t)|^2+|\bar{b}_2(X_t,Y_t)|^2\right)\d t\right\}<\infty
\end{align*}
for some $\delta, t_0>0$. 
Firstly, the Harnack inequality \eqref{HI} implies
\begin{align*}
\E_0\e^{-\Phi_p(t,z,\cdot)}\left(P_t^0\e^{\frac{\varepsilon(|\bar{b}_1|^2+|\bar{b}_2|^2)}{p}}\right)^p(z)\leq \E_0\e^{\varepsilon(|\bar{b}_1|^2+|\bar{b}_2|^2)}.
\end{align*}
So for any $s\in(0,1)$ and $\lambda_s=\frac{\varepsilon}{ps}$, by Jensen's inequality and \eqref{E0p}, we arrive at
\begin{align*}
&\bar{\E}^{G}\exp\left\{\lambda_s \int_0^{s}\left(|\bar{b}_1(X_t,Y_t)|^2+|\bar{b}_2(X_t,Y_t)|^2\right)\d t\right\}\\
&\leq \frac{1}{s}\int_0^{s}\bar{\E}^{G}\exp\left\{\frac{\varepsilon}{p} \left(|\bar{b}_1(X_t,Y_t)|^2+|\bar{b}_2(X_t,Y_t)|^2\right)\right\}\d t\\
&=\frac{1}{s}\int_0^{s}P_t^0\e^{\frac{\varepsilon(|\bar{b}_1|^2+|\bar{b}_2|^2)}{p}}(z)\d t\\
&\leq \frac{1}{s}\int_0^{s}\frac{\d t}{\left\{\E_0\e^{-\Phi_p(t,z,\cdot)}\right\}^{\frac{1}{p}}} \left(\E_0\e^{\varepsilon(|\bar{b}_1|^2+|\bar{b}_2|^2)}\right)^{\frac{1}{p}}<\infty, \ \ z\in\R^2.
\end{align*}
Thus, taking $t_0$ satisfying $\frac{\varepsilon}{pt_0}>(\underline{\sigma}^{-2}+\overline{\sigma}^2)$ and $\delta=\frac{\frac{\varepsilon}{pt_0}}{2(\underline{\sigma}^{-2}+\overline{\sigma}^2)}-\frac{1}{2}$, the proof is completed.
\end{proof}
Next, we give an example in which \eqref{HI} and \eqref{E0p} hold.
\begin{exa}  In \eqref{EH1}, let $A=0$, $M=Q=1$, $b_2=0$ and $b_1(x,y)=-x-y$. Then \eqref{EH0} reduces to
\beq\label{EH3}
\begin{cases}
\d X_t=Y_t\d t, \\
\d Y_t=(-X_t-Y_t)\d t+\d B_t.
\end{cases}
\end{equation}
Firstly, by Theorem \ref{T3.2} (1), \eqref{HI} holds for
\begin{align*}
&\Phi_p(T,z,\bar{z})=\left[C\frac{p}{p-1}\left(\underline{\sigma}^{-2}T\left(\ff 1 {T}+ \ff{1}{T^{2}}+1+ T\right)^2+\overline{\sigma}^{2}T\left(1+ T\right)^2\right)\right]|z-\bar{z}|^2\\
&=c\frac{|z-\bar{z}|^2}{T^3}, \ \ T\in(0,1)
\end{align*}
for some constant $c>0$.

Next, by \cite[Theorem 3.12]{HLWZ}, \eqref{EH3} has a unique invariant nonlinear expectation $\E_0$. Let $\theta^0_s=\underline{\sigma},s\geq 0$ and $\P_{\theta^0}$ be the corresponding probability
as represented in \eqref{rep2}.  Then
\begin{equation}\label{jixian}
\bar{\E}^{G}f(X_t,Y_t)\geq \E_{\P_{\theta^0}}f(X_t,Y_t),\ \ f\in C_b^1(\R^2).
\end{equation}
On the other hand, by \cite[Theorem 3.1(1)]{W2}, under the probability $\P_{\theta^0}$, \eqref{EH3} has a unique invariant measure $\mu_0$:
$$\mu_0(\d x,\d y)=\frac{1}{2\pi\underline{\sigma}^2}\e^{-\frac{|x|^2+|y|^2}{2\underline{\sigma}^2}}\d x\d y.$$
By \cite[Theorem 3.3, Theorem 3.12]{HLWZ},  letting $t$ go to infinity in \eqref{jixian}, we arrive at
$$\E_0f\geq \mu_0(f), \ \ f\in C_b^1(\R^2).$$
Thus, according to \cite[Example 4.3]{WP}, we have
$$\E_0\left(\e^{-\Phi_p(t,z,\cdot)}\right)\geq \mu_0\left(\e^{-\Phi_p(t,z,\cdot)}\right)\geq \e ^{-c}\mu_0(B(z,1\wedge t^{\frac{3}{2}}))\geq \alpha(z)(1\wedge t)^{\frac{3}{2}},\ \ t>0, z\in\R^2$$
for some $\alpha\in C(\R^2)$. Thus, \eqref{E0p} holds for $p>\frac{3}{2}$.
\end{exa}

\beg{thebibliography}{99}

\bibitem{15} L. Denis, M. Hu, S. Peng, \emph{Function spaces and capacity related to a sublinear expectation: application to $G$-Brownian motion pathes}, Potential Anal. 34(2011), 139-161.

\bibitem{Denis} L. Denis and C. Martini, \emph{A theorectical framework for the pricing of contingent claims in the
presence of model uncertainty}, Ann. Appl. Probab. 16(2006), 827-852.


\bibitem{GW} A. Guillin, F.-Y. Wang, \emph{Degenerate Fokker-Planck equations: Bismut formula, gradient estimate and Harnack inequality,} J. Differential Equations 253(2012), 20-40.


\bibitem{HJ} M. Hu, S. Ji, \emph{Stochastic maximum principle for stochastic recursive optimal control problem under volatility ambiguity}, SIAM J. Control Optim. 54(2016), 918-945.

\bibitem{HJ2}  M. Hu, S. Ji, \emph{Dynamic programming principle for stochastic recursive optimal control problem under $G$-framework,} Stochastic Process. Appl.  127(2017), 107-134.

\bibitem{HJPS} M. Hu, S. Ji, S. Peng, Y. Song, \emph{Comparison theorem, Feynman-Kac formula and Girsanov transformation for BSDEs driven by $G$-Brownian motion,} Stochastic Process. Appl. 124(2014), 1170-1195.

\bibitem{HLWZ} M. Hu, H. Li, F. Wang, G. Zheng, \emph{Invariant and ergodic nonlinear
expectations for $G$-diffusion processes,} Electron. Commun. Probab. 20(2015), 1-15.
\bibitem{HP} M. Hu, S. Peng, \emph{On representation theorem of $G$-expectations and paths of $G$-Brownian motion,} Acta Math. Appl. Sin. Engl. Ser. 25(2009), 539-546.

\bibitem{O} E. Osuka, \emph{Girsanov's formula for $G$-Brownian motion,} Stochastic Process. Appl. 123(2013), 1301-1318.

\bibitem{peng2} S. Peng, \emph{$G$-Brownian motion and dynamic risk measures under volatility
uncertainty,} arXiv: 0711.2834.

\bibitem{peng1} S. Peng, \emph{$G$-expectation, $G$-Brownian motion and related stochastic calculus of It\^{o} type,} Stoch. Anal. Appl. 2(2007), 541-567.

\bibitem{peng4} S. Peng, \emph{Nonlinear expectations and stochastic calculus under uncertainty with robust CLT and G-Brownian motion,} Probability Theory and Stochastic Modelling, Springer, 2019.

\bibitem{P} S. Peng, \emph{Theory, methods and meaning of nonlinear expectation theory (in Chinese),} Sci Sin Math 47(2017), 1223-1254.


\bibitem{song} Y. Song \emph{Gradient Estimates for Nonlinear Diffusion Semigroups by Coupling Methods,} https://doi.org/10.1007/s11425-018-9541-6.


\bibitem{Wbook} F.-Y. Wang, \emph{Harnack Inequality and Applications for Stochastic Partial Differential Equations,} Springer, New York, 2013.

\bibitem{W2} F.-Y. Wang, \emph{Hypercontractivity and Applications for Stochastic Hamiltonian Systems,} J. Funct. Anal., 272(2017), 5360-5383.



\bibitem{WP} F.-Y. Wang, \emph{Estimates for invariant probability measures of degenerate SPDEs with singular and path-dependent drifts,} Probab. Theory Related Fields 172(2018), 1181-1214.

\bibitem{WA} F.-Y. Wang, \emph{Integrability conditions for SDEs and semilinear SPDEs,} Ann. Probab. 45(2017), 3223-3265.

\bibitem{WZ1} F.-Y. Wang, X. C. Zhang, \emph{Derivative formula and applications for degenerate diffusion semigroups,} J. Math. Pures Appl., 99(2013), 726-740.


\bibitem{Yang} F.-F. Yang, \emph{Harnack Inequality and Applications for SDEs Driven by $G$-Brownian motion,} arXiv:1808.08712.

\end{thebibliography}

\end{document}